\documentclass{birkau}

\usepackage{amssymb, latexsym, url}

\def\Om{\Omega}

\def\mb{\mathbb}
\def\mbf{\mathbf}

\newcommand{\mc}{\mathcal}
\newcommand{\mcV}{\mathcal{V}}
\newcommand{\mcK}{\mathcal{K}}

\newcommand{\mcQ}{\mathcal{Q}}
\newcommand{\mcT}{\mathcal{T}}

\newcommand{\mcW}{\mathcal{W}}
\newcommand{\mcS}{\mathcal{S}}
\newcommand{\mcL}{\mathcal{L}}

\def\wt{\widetilde}
\def\tV{\widetilde{\mcV}}
\def\tqV{\widetilde{\mcV}^q}

\def\tT{\widetilde{\mcT_t}}

\def\oT{\overline{\mcT_t}}

\def\mbf{\mathbf}

\def\mf{\mathsf}

\newcommand{\bfx}{\mathbf x}
\newcommand{\bfa}{\mathbf a}
\newcommand{\pro}[1]{{#1}^p}

\theoremstyle{plain}
\newtheorem{theorem}{Theorem}[section]
\newtheorem{corollary}[theorem]{Corollary}
\newtheorem{proposition}[theorem]{Proposition}
\newtheorem{lemma}[theorem]{Lemma}

\theoremstyle{definition}
\newtheorem{definition}[theorem]{Definition}
\newtheorem{remark}[theorem]{Remark}
\newtheorem{example}[theorem]{Example}
\newtheorem{problem}[theorem]{Problem}

\numberwithin{equation}{section}

\usepackage[dvipsnames]{color}

\begin{document}

\title[Semilattice sums of algebras]{Semilattice sums of algebras\\ and Mal'tsev products of varieties}

\corrauthor[C. Bergman]{C. Bergman}
\address{ Department of Mathematics\\
Iowa State University\\
Ames, Iowa, 50011, USA}
\email{cbergman@iastate.edu\phantom{,}}

\author[T. Penza]{T. Penza}
\address{Faculty of Mathematics and Information Science\\
Warsaw University of Technology\\
00-661 Warsaw, Poland}
\email{t.penza@mini.pw.edu.pl}

\author[A.B. Romanowska]{A.B. Romanowska}
\address{Faculty of Mathematics and Information Science\\
Warsaw University of Technology\\
00-661 Warsaw, Poland}
\email{aroman@mini.pw.edu.pl}

\thanks{Research of the first author was partially supported by the National Science Foundation under grant
no.~1500235.
The second author's research was supported by the Warsaw University of Technology under grant number 504/04259/1120.}

\keywords{Mal'tsev product of varieties, semilattice sums, prolongation, P\l onka sums,
  Lallement sums, regular and irregular identities, regularization and
  pseudo-regularization of a variety}

\subjclass{08B05, 08C15,  08A05}

\begin{abstract}
The Mal'tsev product of two varieties of similar
algebras is always a quasivariety. We consider the question of when
this quasivariety is a variety. The main result asserts that
if $\mcV$ is a strongly irregular variety with no nullary operations and at least one
non-unary operation, and $\mcS$
is the variety, of the same type as $\mcV$, equivalent to the variety of
semilattices, then the Mal'tsev product $\mcV \circ \mcS$ is a variety.
It consists precisely of semilattice sums of algebras in $\mcV$.
We derive an equational base for the product from an
equational base for $\mcV$. However, if
$\mcV$ is a regular variety, then the Mal'tsev product may not be a
variety. We discuss  various applications of the main result,
and examine some detailed representations of algebras in $\mcV
\circ \mcS$.
\end{abstract}

\maketitle

Let $\mc{K}$ be a quasivariety of $\Om$-algebras, and $\mc{Q}$ and $\mc{R}$, two of its
subquasivarieties. Assume additionally that $\mc{R}$-algebras are idempotent. Then the
\emph{Mal'tsev product, $\mc{Q} \circ_{\mc{K}} \mc{R}$, of $\mc{Q}$ and $\mc{R}$ relative
  to $\mc{K}$,} consists of $\mc{K}$-algebras $A$, with a congruence $\theta$, such that
$A/\theta$ is in $\mc{R}$, and each $\theta$-class $a/\theta$ is in $\mc{Q}$. Note that by
the idempotency of $\mc{R}$, each $\theta$-class is always a subalgebra of $A$. If
$\mc{K}$ is the variety of all $\Om$-algebras, then the Mal'tsev product
$\mc{Q} \circ_{\mcK} \mc{R}$ is called simply the \emph{Mal'tsev product} of $\mc{Q}$ and
$\mc{R}$, and is denoted by $\mc{Q} \circ \mc{R}$. It follows from results of  Mal'tsev
\cite{M67} that the  Mal'tsev product $\mc{Q} \circ_{\mc{K}} \mc{R}$ is a quasivariety.

In this paper we are interested in Mal'tsev products $\mcV \circ \mcS$ in which $\mcV$ is
a variety of $\Om$-algebras and $\mcS$ is the variety, of the same type as $\mcV$,
equivalent to the variety of semilattices.  A member of $\mcV \circ \mcS$ is a disjoint
union of $\mcV$-algebras over a homomorphic image, which has the structure of a
semilattice.  Such algebras are known as
semilattice sums of $\mcV$-algebras. The product $\mcV \circ \mcS$ is always a
quasivariety. However, until recently it was not known under what conditions it is a
variety. The main result of this paper (Theorem~\ref{T:varqvar}) asserts that if $\mcV$ is a
strongly irregular variety of a type with no nullary operations and at least one
non-unary operation, then the Mal'tsev product
$\mcV \circ \mcS$ is a variety. We also show how to build an equational base for this
variety from an equational base for~$\mcV$. In addition we provide an example showing that in the
case that $\mcV$ is  regular, $\mcV\circ \mcS$ does not need to be a variety.

Next we present examples of  (relative) Mal'tsev products which exhibit different aspects
of such products. In fact, two special examples inspired us to investigate semilattice
sums in general: one concerning certain semilattice sums of lattices and one involving
certain semilattice sums of Steiner quasigroups. (See Section~\ref{sec:examples}.)

As a converse process to the decomposition of an algebra $A$ from $\mcV \circ \mcS$ into
a semilattice sum of $\mcV$-algebras, some construction techniques are available to
recover an algebra $A$ from its summands and a semilattice quotient. Such techniques
include, for example, P\l onka sums, and the more general Lallement sums
\cite[Ch.~4]{RS02}. We recall some of the basic constructions of this type, and discuss
their applicability for representing algebras in the Mal'tsev product $\mcV \circ \mcS$.

In \cite{M67}, Mal'tsev proved that if $\mcK$ is a congruence permutable quasivariety with
a constant term whose value is always idempotent, then the relative Mal'tsev product of
any two subvarieties of $\mcK$ is a variety. This result was extended by Iskander
\cite{I84} to weakly congruence permutable varieties of algebras.

We begin with introductory information on semilattice sums. In
Sections~\ref{S:subvarieties} and~\ref{S:free} we provide some insight into the
quasivariety of semilattice sums of $\mcV$-algebras. In particular we construct, for any
variety $\mcV$, a supervariety $\pro{\mcV}$, called the \emph{prolongation of $\mcV$,}
containing $\mcV\circ \mcS$. We prove (Proposition~\ref{P:freetp}) that free
$\pro{\mcV}$-algebras lies in $\mcV \circ \mcS$.

The main result, Theorem~\ref{T:varqvar}, asserts that if $\mcV_t$ is a variety satisfying
an identity of the form $t(x,y)=x$ (i.e., a strongly irregular variety, see
Definition~\ref{defn:regular_id}), then $\mcV_t \circ \mcS = \pro{\mcV_t}$. As a
consequence, $\mcV_t \circ \mcS$ is again a variety. The proof is divided into two
parts. In Section~\ref{S:Tt} we treat the special case of the variety $\mcT_t$ of
$\Omega$-algebras defined by a single identity $t(x,y)=x$. Then in Section~\ref{S:Vt} we
leverage that observation to obtain the result for an arbitrary strongly irregular
variety. We also illustrate the necessity of strong irregularity with an example showing
that $\mcS \circ \mcS$ fails to be a variety.

Section~\ref{sec:examples} discusses a number of examples of Mal'tsev products and
relative Mal'tsev products. Relative products are often more interesting, as they may be
used to represent algebras in more familiar varieties.  The paper finishes with a
discussion concerning the detailed representation of algebras in the Mal'tsev product
$\mcV_t \circ \mcS$. Here a version of a construction of algebras called the Lallement sum
(\cite{RS85}) is used. In particular, certain very special Lallement sums provide a good
detailed representation for algebras in $\mcV_t \circ \mcS$, in the case when
$\mcV_t$-algebras have units.

We use notation and conventions similar to those of \cite{RS85, RS02}. For details and
further information concerning quasivarieties and Mal'tsev products of quasivarieties we
refer the reader to \cite{M67} and \cite{M71}, and then also to \cite{B18} and
\cite[Ch.~2]{RS02}; for universal algebra, see \cite{CB12} and \cite{RS02}; in particular,
for methods of constructing algebras as semilattice sums see \cite{PR92} and
\cite[Ch.~4]{RS02};  for semigroup theory, see \cite{CP61, H76, H95}; for
barycentric algebras and convex sets, we refer the reader to the monographs \cite{RS85,
  RS02}.

\section{Background}

Let $\tau\colon \Omega \rightarrow \mathbb{N}$ be a similarity type of $\Omega$-algebras with
no nullary operation symbols and containing at least one non-unary operation symbol. Such
a similarity type is called \emph{plural} \cite{RS02}. For a positive integer $n$, let
$T_{n}$ be the set of all $n$-ary $\Om$-terms containing precisely $n$ different
variables. They may be denoted $x_1, x_2, \dots x_n$ or by some other letters like $y$ or
$z$.

\begin{definition}\label{defn:regular_id}
  An identity is called \emph{regular} if precisely the same variables appear on both
  sides. At the opposite extreme, an identity of the form $t(x,y)=x$, for some $t\in T_2$
  is called \emph{strongly irregular.} A variety is regular if it satisfies only regular
  identities. It is strongly irregular if it satisfies at least one strongly irregular
  identity.
\end{definition}

Let $t\in T_2$. We denote by $\mcT_t$ the variety defined by the
single identity $t(x,y)=x$. We use the notation $\mcV_t$ to indicate a typical subvariety
of $\mcT_t$, that is, an arbitrary variety satisfying the strongly irregular identity
$t(x,y)=x$. We often find it convenient to write $x\cdot y$ (or $x\cdot_t y$) in place of
$t(x,y)$.

Let us note that the variety of  semilattices satifies all the regular identities
of the type with one binary operation. 

\begin{definition}
  The variety $\mcS$ of \emph{$\Omega$-semilattices} is defined to be the variety of
  $\Omega$-algebras satisfying all regular identities.
\end{definition}

It is fundamental to the subject that in a plural similarity type, $\mcS$ is the unique
variety that is equivalent to semilattices. Fix any $t\in T_2$ and write $x\cdot y$ in
place of $t(x,y)$. Then for any $S\in \mcS$, the algebra $\langle S,\cdot\rangle$ is a
semilattice. Conversely, for any basic $n$-ary operation symbol $\omega \in \Omega$ we
have that
\begin{equation}\label{E:sl}
x_1 \dots x_n \omega = x_1\cdot \dots \cdot x_n
\end{equation}
holds in $S$.

\begin{definition}
Let $\mcV$ be a variety of $\Omega$-algebras.
An $\Om$-algebra $A$ is a
\emph{semilattice sum} of $\mcV$-algebras if $A$ has a congruence $\varrho$ such that the
quotient $S =  A/\varrho$ is an $\Om$-semilattice (or briefly a semilattice) and every
congruence block is a $\mcV$-algebra.
\end{definition}

Note that since the quotient $A/\varrho$ is a semilattice, it is idempotent. Hence the
congruence blocks of $\varrho$ are always subalgebras, and $A$ is a disjoint sum of these
subalgebras. We will denote such a sum by $\bigsqcup_{s \in S} A_s$, where the $A_s$'s are
the congruence blocks of $\varrho$ and $S$ is (isomorphic to) $A/\varrho$. The class of
semilattice sums of $\mcV$-algebras forms the Mal'tsev product $\mcV \circ \mcS$ of the
varieties $\mcV$ and $\mcS$, and is known to be a quasivariety. (See Mal'tsev~\cite{M67,
  M71} as well as Bergman~\cite{B18} and also Romanowska, Smith~\cite[Ch.~3]{RS02}.)
Whenever $A$ is a semilattice sum of $\mcV$-algebras, it suffices to take $\varrho$ to be
the semilattice replica congruence of $A$, i.e., the smallest congruence of $A$ with the
corresponding quotient a semilattice. (See~\cite{B18}, just before Lemma~8.) Furthermore,
in the case that $\mcV$ is a strongly irregular variety, the semilattice replica
congruence of $A$ is the unique congruence $\varrho$ providing a decomposition of $A$ into
a semilattice sum of $\mcV$-subalgebras (see~\cite[\S 3.3]{RS02}).

In this paper we are especially interested in semilattice sums of $\mcT_t$-algebras. Since
$\mcT_t$ is a strongly irregular variety, each algebra $A$ in $\mcT_t \circ \mcS$
decomposes into a semilattice sum $\bigsqcup_{s \in S} A_s$ of $\mcT_t$-algebras $A_s$ in
a unique way over the quotient $S$ of $A$ given by its semilattice replica congruence.

The variety $\mcT_t$ contains other irregular varieties. In particular, if $\mcV_t$ is a
subvariety of $\mcT_t$, then it is known that $\mcV_t$ is defined by a set
$\Sigma_t$ consisting of a set $\Gamma_t$ of regular identities and the single identity
$t(x,y) = x$. (See e.g.~\cite{PR92} and \cite[Ch.~4]{RS02}.) The Mal'tsev product
$\mcV_t \circ \mcS$ consists of algebras which are semilattice sums of
$\mcV_t$-subalgebras. Moreover it contains $\mcV_t$ and $\mcS$ as subvarieties, and is
contained in the quasivariety $\mcT_t \circ \mcS$. The smallest variety
$\mathsf{V}(\mcV_t, \mcS)$ containing both $\mcV_t$ and $\mcS$ is called the
\emph{regularization} of $\mcV_t$ and is denoted by $\tV_t$.  (See
e.g.~\cite[Ch.~4]{RS02}.) The smallest quasivariety $\mathsf{Q}(\mcV_t, \mcS)$ containing
both $\mcV_t$ and $\mcS$ is called the \emph{quasi-regularization} of $\mcV_t$ and is
denoted by $\tqV_t$. (See \cite{BR96}.)  Note that $\mathsf{V}(\mcV_t, \mcS)$ and
$\mathsf{Q}(\mcV_t, \mcS)$ are not necessarily equal.  For each $t(x,y) \in T_2$, we have
the following chain of quasivarieties:
\begin{equation*}
\mcV_t,  \mcS  \subseteq  \tqV_t \subseteq  \tV_t \subseteq \mcV_t \circ \mcS \subseteq \mcT_t \circ \mcS.
\end{equation*}
In particular, if the set $\Gamma_t$ of regular identities is empty, then $\mcV_t$ coincides with $\mcT_t$.

\section{Some varieties related to the quasivariety $\mcT_t \circ \mcS$}\label{S:subvarieties}

Recall that, for $x \cdot y = t(x,y) \in T_2$, the subvariety $\mcV_t$ of $\mcT_t$ is defined by a set $\Sigma_t$ consisting of the identity $x \cdot y = x$ and a set of regular identities $\Gamma_t$. The best known subvariety of $\mcV_t \circ \mcS$, different from $\mcV$ and $\mcS$,  is the regularization $\tV_t$ of the variety $\mcV_t$.
We will briefly recall the basic facts concerning the regularization.

The variety $\mcV_t$ can be viewed as a category in which the morphisms are simply the
homomorphisms between $\mcV_t$-algebras. Any ordered set, $S$, can be viewed as a category
whose objects are the elements of the set. If $r, s \in S$ then there is a unique morphism
from $r$ to $s$ if $r\geq s$. Otherwise there are no morphisms from $r$ to $s$. Of course,
 any semilattice can be ordered by considering the basic operation to be the infimum.

Suppose that $S$ is a semilattice and there is a covariant functor
\begin{equation*}
F\colon S  \rightarrow \mcV_t;\  (r \geq s) \mapsto (\varphi_{r,s}\colon A_r \rightarrow A_s).
\end{equation*}
Then the disjoint sum $\bigsqcup_{s \in S} A_s$ may be viewed as the semilattice-ordered system
\begin{equation}\label{eq:1}
\bigl(\{A_s \mid s\in  S\};\, \{\varphi_{r,s}\mid  r,s \in S, r \geq s\}\bigr),
\end{equation}
where the $\Om$-operations on the sum are defined by
\begin{equation*}
\omega\colon A_{s_1} \times \dots \times A_{s_n} \rightarrow A_s ; (a_{s_1}, \dots, a_{s_n})
\mapsto a_{s_1}\varphi_{s_1,s} \dotsm a_{s_n}\varphi_{s_n,s}\, \omega
\end{equation*}
for each ($n$-ary) $\omega \in \Om$, where each $s_i \in S$ and
$s = s_1\cdot \dots \cdot s_n$.  The resulting semilattice sum is called the \emph{P\l
  onka sum} of the algebras $A_s$, and is denoted by $\sum_{s\in S} A_s$. (See \cite{P67},
\cite{PR92}, \cite[Ch.~4]{RS02} .)

It is a remarkable fact that the class of P\l onka sums of $\mcV_t$-algebras forms the
\emph{regularization} $\tV_t$ of the variety $\mcV_t$. The class $\tV_t$ is the join (in
the lattice of varieties) of $\mcV_t$ and $\mcS$. Thus it satisfies precisely the regular
identities true in $\mcV_t$.  Its axiomatization follows from P\l onka's theorem (see
\cite{PR92}, \cite[Ch.~4]{RS02}). The variety $\tV_t$ is defined by the regular identities
$\Gamma_t$ of $\Sigma_t$ and the following identities (P1) -- (P5): 
\begin{itemize}
\item[(P1)] $x \cdot x = x$,
\item[(P2)] $x \cdot (y\cdot z) = (x\cdot y) \cdot z$,
\item[(P3)] $x \cdot (y\cdot z) = x \cdot (z\cdot y)$,
\item[(P4)] $y \cdot x_1 \dots x_n \omega = y \cdot x_1 \cdot \dots \cdot x_n$,
\item[(P5)] $x_1 \dots x_n \omega \cdot y = (x_1 \cdot y ) \dots (x_n \cdot y)\, \omega,$
\end{itemize}
where $\omega$ ranges through all symbols in $\Om$.  Let us note that if $\Gamma_t$ is
empty, then the regularization $\tV_t$ becomes the regularization $\tT$, and is defined by
the identities (P1) -- (P5).

By results of \cite{BR96} it is known that the quasi-regularization $\tqV_t$ of $\mcV_t$
consists of P\l onka sums of $\mcV_t$-algebras as in~\eqref{eq:1} in which the morphisms
$\varphi_{r,s}$ are all injective. It coincides with the class
$\mathsf{I}\mathsf{S}\mathsf{P}(\mcV_t \cup \mcS)$ of (isomorphic copies of) subalgebras
of products of algebras in $\mcV_t \cup \mcS$, and is defined, relative to $\tV_t$, by the quasi-identity
\begin{equation*}
(x \cdot y = x \mathbin{\&} y \cdot x = y \mathbin{\&} x \cdot z = z \cdot x = z \mathbin{\&} y
\cdot z = z \cdot y = z) \rightarrow (x = y).
\end{equation*}

The congruence $\varrho$ of a $\tV_t$-algebra $A$ providing its decomposition into the P\l onka sum of its $\mcV_t$-summands can be obtained explicitly as
\begin{equation}\label{E:partition}
\varrho = \{(a,b) \mid a \cdot b = a\,\, \mbox{and}\,\, b \cdot a = b\}.
\end{equation}

Now, let $\oT$ denote the variety defined by the identities (P1) -- (P4)
above. It follows immediately that $\tT \subseteq \overline{\mcT_t}$. For a subvariety
$\mcV_t$ of $\mcT_t$, let $\overline{\mcV_t}$ be defined by $\Gamma_t$ together with (P1)
-- (P4). As observed by C.~Bergman and D. Failing \cite{BF15},
$\overline{\mcV_t} \subseteq \mcV_t \circ \mcS$. That is to say, each algebra $A$ in
$\overline{\mcV_t}$ is a semilattice sum of  $\mcV_t$-subalgebras. Moreover, the
decomposition is given by the same formula \eqref{E:partition} as is the case for the
regularization. However, the $\overline{\mcV_t}$-algebras are not necessarily P\l onka
sums of $\mcV_t$-algebras, though the reduct $(A, \cdot_t)$ of each
$\overline{\mcV_t}$-algebra $A = \bigsqcup_{s \in S}A_s$ is a left-normal band and is the P\l
onka sum $\sum_{s \in S}A_s$ of the left-zero band reducts $(A_s, \cdot)$. We will call
the variety $\overline{\mcV_t}$ the \emph{pseudo-regularization} of the variety $\mcV_t$.
If the operation $x \cdot y$ of an $\Om$-algebra $(A, \Om)$ satisfies the conditions (P1)
-- (P4), then it is called a \emph{pseudopartition operation.} If it
satisfies all of (P1) -- (P5) it is called a \emph{partition operation.}

The third variety we are interested in will play an essential role in our main results.

\begin{definition}\label{D:eqbasprod}
Let $\sigma$ be an identity of type $\tau$ of the form
  \begin{equation*}
    y_1\ldots y_nu = y_1\ldots y_n v,
\end{equation*}
where not all the variables $y_i$ necessarily appear in $u$ and $v$.  

For each $m > 0$, define
\begin{align*}
    \sigma^{p}_m = \bigl\{\, &\bfx r_1\ldots \bfx r_n u = \bfx r_1\ldots \bfx r_n v \mid \\ &r_1,\dots, r_n \in T_m , \bfx r_i = x_1 \ldots x_m r_i
     \, \bigr\}.
\end{align*}
Let $\sigma^p$ be the union of all $\sigma^{p}_{m}$ for $m > 0$.   For a set $\Sigma$ of
identities, define its \emph{prolongation} to be 
\begin{equation*}
    \Sigma^p = \bigcup_{\sigma \in \Sigma} \sigma^p.
\end{equation*}
\end{definition}

The intent of Definition~\ref{D:eqbasprod} is that, for every positive integer $m$, each
variable $y_i$ in $\sigma$ is replaced by a term of the form $x_1\ldots x_m
r_i$. However, since $r_i \in T_m$, each term $r_i$ must contain precisely all of the
variables $x_1,\dots,x_m$. Thus even if the identity $\sigma$ is irregular, every member
of $\sigma^p$ will be a regular identity.

\begin{definition}
  Let $\mcV$ be a variety of $\Omega$-algebras, and let $\operatorname{Id}(\mcV)$ be the
  set of all equations that hold in $\mcV$. We define \emph{the prolongation of $\mcV$,}
  denoted $\pro{\mcV}$, to be the variety of all algebras satisfying
  $\pro{\operatorname{Id}(\mcV)}$. Put succinctly,
  $\pro{\mcV} = \operatorname{Mod}(\operatorname{Id}(\pro{\mcV)})$.
\end{definition}

In fact, in this definition, it suffices to replace $\operatorname{Id}(\mcV)$ by any
equational base for $\mcV$. This will follow from Corollary~\ref{C:vp_base}.

\begin{lemma}\label{L:VoS}
  Let $\sigma$ be an identity that holds in a variety $\mcV$. Then the quasivariety $\mcV
  \circ \mcS$ satisfies the identities of $\sigma^{p}$.
\end{lemma}

\begin{proof}
  Assume $\sigma$ is as in Definition~\ref{D:eqbasprod} and let $A$ be in
  $\mcV \circ \mcS$ with semilattice replica congruence $\varrho$. Let $m$ be a positive
  integer and $r_1,\dots, r_n \in T_m$.  Choose $a_1,\dots,a_m \in A$ and write
  $\mathbf{a} r_i $ as a shorthand for $a_1\ldots a_m r_i$. We must show that
  $\bfa r_1 \ldots \bfa r_n u = \bfa r_1 \ldots \bfa r_n v$ holds in $A$.

For any $i,j \leq m$, the identity $x_1\ldots x_m r_i = x_1\ldots x_m r_j$ is
regular. Since $A/\varrho$ is a semilattice, it must satisfy this identity. Hence
\begin{equation*}
(a_1\ldots a_m r_i,\, a_1\ldots a_m r_j) \in \varrho.
\end{equation*}
Thus all of the elements $\bfa r_i$, for $i=1,\dots,n$, lie in the same congruence
block. By assumption, each congruence block is a member of $\mcV$, so it satisfies the
identity $\sigma$. From this we obtain $\bfa r_1 \ldots \bfa r_n u = 
\bfa r_1 \ldots \bfa r_n v$ as desired.
\end{proof}

\begin{proposition}\label{P:VoS}
For any variety $\mcV$,
\begin{equation}\label{E:VSinP}
\mcV \circ \mcS \subseteq \pro{\mcV}.
\end{equation}
\end{proposition}

\section{Free algebras in varieties $\pro{\mcV}$}\label{S:free}

Basic properties of free algebras
in sub(quasi)varieties of $\pro{\mcV}$ follow from basic properties of free algebras in
prevarieties of $\Om$-algebras (i.e. classes of similar algebras closed under subalgebras
and direct products) as given for example in \cite[\S3.3]{RS02}.
If $\mcQ$ and
$\mcQ'$ are two subquasivarieties with $\mcQ' \leq \mcQ$, then the free
$\mcQ'$-algebra $X\mcQ'$ over $X$ is isomorphic to the $\mcQ'$-replica of the algebra
$X\mcQ$. (See also \cite[Thm.~4.28]{CB12}.) In particular, if $\mcV$ is a regular
variety, then the semilattice replica $X\mcV/\varrho$ of $X\mcV$ is
isomorphic to the free semilattice $X\mcS$:
\begin{equation}\label{E:freesl}
X\mcV/\varrho \cong X\mcS.
\end{equation}
Recall also that the free semilattice $X\mcS$ over $X$ is isomorphic to the semilattice
$P_{f}(X)$ of all finite non-empty subsets of $X$ under union.

For an $\Omega$-term $w$,  $\mf{var}(w)$ will denote the set of variables that actually
appear in  $w$. Suppose that $\mcV$ is a regular variety and let $a$ be an element of
$X\mcV$. There is a term $w$ with  $\mf{var}(w) = \{x_1,\dots, x_n\} \subseteq X$ so that
$a=x_1,\dots,x_nw$. If $a=y_1,\dots,y_mv$ for some other term $v$ and $y_1,\dots, y_m
\in X$,  then by the freeness of $X\mcV$ the identity $x_1,\dots,x_nw = y_1,\dots,
y_mv$ holds in $\mcV$. By regularity we must have $\mf{var}(w) = \mf{var}(v)$, which we can take
as a definition of $\mf{var}(a)$. As is customary, we often blur the distinction between
an element of the free algebra and a representing term.

\begin{lemma}\label{L:freereg}
  Let $\mcV$ be a regular variety and let $\varrho$ be the
  semilattice replica congruence of the free algebra $X\mcV$ over $X$. Then for $w,v \in X\mcV$
  \begin{equation*}
(w,v) \in \varrho \iff \mf{var}(w) = \mf{var}(v).
\end{equation*}
\end{lemma}
\begin{proof}
Let $\theta$ be the binary relation on $A = X\mcV$ defined by
\begin{equation*}
(w,v) \in \theta \iff \mf{var}(w) = \mf{var}(v).
\end{equation*}

It is easy to see that the relation $\theta$ is a congruence relation, and is the kernel of
the homomorphism
\begin{equation*}
h \colon X\mcV \rightarrow P_f(X);\ x_1 \dots x_n u \mapsto \mf{var}(u)
\end{equation*}
onto the semilattice of all finite subsets of $X$ under union.  Now by universality of
replication (see \cite[Lemma~3.3.1]{RS02}), it follows that there is a unique homomorphism
$\overline{h}\colon X\mcV/\varrho \rightarrow P_f(X)$ such that the composition
$(\mf{nat}\varrho)\overline{h}$ equals $h$. In particular, under this composition
$u \mapsto u/\varrho \mapsto \mf{var}(u)$.  By \eqref{E:freesl}, $\overline{h}$ is an
isomorphism. Hence $\theta = \varrho$.
\end{proof}

In particular, for any variety $\mcV$, the variety $\pro{\mcV}$ is regular. Thus
\begin{equation}\label{E:replP}
X\pro{\mcV}/\varrho \cong X\mcS.
\end{equation}

\begin{proposition}\label{P:freetp}
  Let $\mcV$ be a variety with equational base $\Sigma$. Let $\mcW$ be the variety defined
  by the prolongation $\pro{\Sigma}$ (see Definition~\ref{D:eqbasprod}). Then the free
  $\mcW$-algebra  $A=X\mcW$ is a semilattice
of $\mcV$-algebras. More precisely, $A = \bigsqcup_{s \in S} A_s $, where $S = X\mcS$ is
the free semilattice over $X$ and all $A_s$ are members of $\mcV$.

In particular, by taking $\Sigma$ to be the set of all identities holding in $\mcV$,
\begin{equation*}
  A = X\pro{\mcV} \in \mcV \circ \mcS.
\end{equation*}
\end{proposition}

\begin{proof}
  By \eqref{E:replP}, we already know that $A/\varrho$ is isomorphic to $S = X\mcS$. We
  complete the proof by showing that each $A_s$ satisfies the identities of
  $\Sigma$. Consider any identity $y_1 \dots y_n u = y_1 \dots y_n v$ of $\Sigma$. Let
  $a_1, \ldots, a_n$ be elements of $A_s$. Since $A_s$ is a $\varrho$-class, it follows
  that all $a_1, \ldots, a_n$ are $\varrho$-equivalent. Thus, by Lemma~\ref{L:freereg},
  $\mf{var}(a_1) = \dots = \mf{var}(a_n) = \{x_1,\dots, x_m\}$ for some $m>0$. This means
  that each $a_i$ may be represented as a term
  $x_1 \dots x_m r_i \in T_m$. Thus
  \begin{align*}
    a_1 \dots a_n u = &(x_1 \dots x_m r_1) \, \ldots \, (x_1 \dots x_m r_{n}) u =\\
    &(x_1 \dots x_m r_{1}) \, \ldots \, (x_1 \dots x_m r_{n}) v =
    a_1  \dots a_n v,
  \end{align*}
  since the middle equality is an identity in $\Sigma^p$.
\end{proof}

\begin{corollary}\label{C:vp_base}
  Let $\mcV$ be a variety with equational base $\Sigma$. Then $\pro{\Sigma}$ is an
  equational base for $\pro{\mcV}$.
\end{corollary}

\begin{proof}
  Let $\mcW$ be the variety defined by $\pro{\Sigma}$ and let $\Delta$ be the set of all
  identities holding in $\mcV$. Applying Proposition~\ref{P:freetp} to both $\Sigma$ and
  $\Delta$ we see that for any set $X$, both $X\mcW$ and $X\pro{\mcV}$ are free algebras
  in the quasivariety $\mcV \circ \mcS$. But free algebras are unique and furthermore
  determine the variety. So $\mcW = \pro{\mcV}$.
\end{proof}

\begin{corollary}\label{C:gen}
The variety $\pro{\mcV}$ is generated by the quasivariety $\mcV \circ \mcS$.
\end{corollary}

In the next two sections we will show that if $\mcV$ is a strongly irregular variety, then the
variety $\pro{\mcV}$ coincides with the Mal'tsev product $\mcV \circ \mcS$.

\section{The Mal'tsev product $\mcT_t \circ \mcS$}\label{S:Tt}

In this section we clarify the relationship between the quasivariety
$\mcT_t \circ \mcS$ and the variety $ \pro{\mcT_t}$.

Following Proposition \ref{P:freetp}, we assume that the free $\pro{\mcT_t}$-algebra
$A = X\pro{\mcT_t}$ over $X$ is the semilattice sum $A = \bigsqcup_{s \in S} A_s $, where
$S$, the semilattice replica of $A$, is the free semilattice over $X$, and the summands
$A_s$ are in $\mcT_t$. The elements of $A_s$ will be denoted by small letters
$a_s, b_s, ...$ with the same index as in $A_s$, and we will write $xy$ for $t(x,y)$. Note
that if $a_r \in A_r$ and $b_s \in A_s$ then the semilattice structure dictates that $a_r
b_s \in A_{rs}$.

\begin{lemma}\label{L:congrfree}
Let $\theta$ be a congruence of the free $\pro{\mcT_t}$-algebra $A =
X\pro{\mcT_t}$. Assume that there is a pair $(a'_r, b'_s)$ of elements of $A$, where $r, s
\in S$, such that $(a'_r, b'_s) \in \theta$. Then for any $a_r \in A_r$ and any $b_s \in
A_s$
\begin{equation*}
(b_s , b_s a_r) \in \theta.
\end{equation*}

\end{lemma}
\begin{proof}

From our assumptions, $b_s, b'_s \in A_s$, $a_r, a'_r \in A_r$ and $(b_sa'_r), (a_r b'_s) \in
A_{rs}$. Each of $A_r$, $A_s$, and $A_{rs}$ lie in $\mcT_t$, so they satisfy the identity
$xy=x$. Consequently
\begin{equation*}
  b_s = b_s b'_s, \quad a_r = a_r a'_r, \quad (b_sa'_r)(a_rb'_s) = b_sa'_r.
\end{equation*}
Combining these equalities with the assumption that $b'_s \mathrel{\theta} a'_r$ we obtain
\begin{equation*}
  b_s a_r = (b_sb'_s)(a_ra'_r) \mathrel{\theta} (b_sa'_r)(a_rb'_s) = (b_s a'_r)
  \mathrel{\theta} (b_s b'_s) = b_s.
\end{equation*}
\end{proof}

\begin{remark}\label{R:congrfree}
Suppose $r\neq s$ in $S$. If $(a'_r, b'_s) \in \theta$, then each $a_r$ and each $b_s$ is
$\theta$-related to some element of $A_{rs}$. Moreover, if additionally $q < rs < r$ and
also $(c'_{rs}, d'_q) \in \theta$ for some $c'_{rs} \in A_{rs}$ and $d'_q \in A_q$, then
each element of $A_{rs}$ is $\theta$-related to some element of $A_q$, and in particular,
$a_r \, \theta \, a_r c_{rs} \, \theta \, a_r c_{rs} \cdot d_q$ for any $a_r, c_{rs},
d_q$.
\end{remark}

The next lemma follows from \cite[Lemma~4.66]{MMT87}.

\begin{lemma}\label{L:3-perm}
Let $A$ be an $\Om$-algebra and let $\alpha$ and $\beta$ be congruences of~$A$. Then the
following conditions are equivalent.
\begin{itemize}
\item[(a)] $\beta \circ \alpha \circ \beta \subseteq \alpha \circ \beta  \circ \alpha$,
\item[(b)] $\alpha \vee \beta = \alpha \circ \beta  \circ \alpha$.
\end{itemize}
\end{lemma}

\begin{lemma}\label{L:3permrt}
  Let $\theta$ be a congruence of the free $\pro{\mcT_t}$-algebra $A = X\pro{\mcT_t}$, and
  let~$\varrho$ be the semilattice replica congruence of $A$. Then
\begin{equation}
\theta \vee \varrho = \theta \circ \varrho \circ \theta.
\end{equation}
\end{lemma}

\begin{proof}
  Suppose that
  \begin{equation*}
a_r \ \varrho \ a'_r \ \theta \ b'_s \ \varrho \ b_s.
\end{equation*}
Since $A_r$ and $A_s$ lie in $\mcT_t$, and $A/\varrho$ is commutative
\begin{equation*}
a_r = a_ra'_r \mathrel{\theta} a_r b'_s \mathrel{\varrho} b_s a'_r
\mathrel{\theta}  b_sb'_s = b_s.
\end{equation*}
Therefore
\begin{equation*}
\varrho \circ \theta \circ \varrho \, \subseteq \, \theta \circ \varrho \circ\theta.
\end{equation*}
By Lemma \ref{L:3-perm}, it follows that
\begin{equation*}
\theta \vee \varrho  = \theta \circ \varrho \circ\theta.
\end{equation*}
\end{proof}

Let $\theta$ be a congruence of a $\pro{\mcT_t}$-algebra $A$, and $\varrho$ the
semilattice replica congruence of $A$. Let $\psi$ be the join of the congruences $\theta$
and $\varrho$.  It is well known that $(A/\theta)/(\psi/\theta) \cong A/\psi$, and since
$\psi \geq \varrho$, it follows that $A/\psi$ is a member of the variety $\mcS$ of
$\Om$-semilattices.

\begin{lemma}\label{L:equivlz}
Let $\theta$ be a congruence of a $\pro{\mcT_t}$-algebra $A$, and let $\psi = \theta \vee
\varrho$. Then a congruence class $(a_r/\theta)/(\psi/\theta)$  of $A/\theta$ satisfies
the identity $x \cdot y = x$ if and only if, for any $b_s \in A$
\begin{equation}\label{E:equivlz}
(a_r, b_s) \in \psi \implies (a_r b_s, a_r) \in \theta.
\end{equation}
\end{lemma}
\begin{proof}
Let $a_r/\theta$ and $b_s/\theta$ be members of $A/\theta$. Then
\begin{equation*}
 (a_r/\theta, b_s/\theta) \in \psi/\theta \,\, \iff \,\, (a_r,b_s) \in \psi.
\end{equation*}
On the other hand
\begin{equation*}
(a_r/\theta) \cdot (b_s/\theta) = a_r/\theta \,\, \iff \,\, (a_r b_s, a_r) \in \theta.
\end{equation*}
Hence the assertion that $(a_r/\theta)/(\psi/\theta)$ satisfies $x \cdot y = x$ is equivalent to~\eqref{E:equivlz}.
\end{proof}

\begin{proposition}\label{P:quotient}
  Let $\theta$ be a congruence of the free $\pro{\mcT_t}$-algebra $A = X\pro{\mcT_t}$
  over~$X$, and set $\psi=\theta \vee \varrho$. Then the quotient $B = A/\theta$ of $A$ is
  a semilattice sum of $\mcT_t$-subalgebras. In particular, $\psi/\theta$ is the
  semilattice replica congruence of $A/\theta$, and $B$ is the semilattice sum of the
  $\psi/\theta$-classes $(a_r/\theta)/(\psi/\theta)$ as $a_r$ ranges over $A$.
\end{proposition}

\begin{proof}
Since $\psi \geq \varrho$, it follows that $A/\psi \in \mcS$. Thus we only need to show
that for each $a_r/\theta \in A/\theta$, the congruence class $(a_r/\theta)/(\psi/\theta)$
satisfies $x \cdot y = x$. By Lemma \ref{L:equivlz}, this is equivalent to the
condition~\eqref{E:equivlz}.

So our aim now is to verify the implication \eqref{E:equivlz}. For $a_r \in A_r$ and
$b_s \in A_s$, let $(a_r, b_s) \in \psi$.  By Lemma \ref{L:3permrt}, we know that
\begin{equation*}
\psi = \theta \vee \varrho =\theta \circ \varrho \circ\theta.
\end{equation*}
Thus there are elements $c$ and $d$ in $A$ such that $a_r \mathrel{\theta} c
\mathrel{\varrho} d \mathrel{\theta} b_s$. But $c \mathrel{\varrho} d$ implies $c=cd$. Hence
$a_r \mathrel{\theta} c = cd
\mathrel{\theta} a_rb_s$.
\end{proof}

Let us note that $A/\theta$ is isomorphic to the semilattice sum of the $\psi$-classes
$a/\psi$ of $A$ over the semilattice $A/\psi$, where each $a/\psi$ is the union of
$\theta$-classes $b/\theta$ such that $(a,b) \in \psi$ (or equivalently $(a/\theta,
b/\theta) \in (\psi/\theta)$).

As a direct corollary of Proposition \ref{P:quotient} one obtains the following theorem.

\begin{theorem}\label{T:varPt}
The quasivariety $\mcT_{t} \circ \mcS$ and the variety $\pro{\mcT_t}$ coincide.
In particular, the Mal'tsev product $\mcT_{t} \circ \mcS$ is a variety, and $\{t(x,y) = x\}^p$ is its equational base.
\end{theorem}

\begin{proof}
  We have already observed that $\mcT_t \circ \mcS \subseteq \pro{\mcT_t}$. Conversely, if
  an algebra~$B$ is a member of $\pro{\mcT_t}$, then there is a set $X$ so that $B$ is a
  homomorphic image of $X\pro{\mcT_t}$. By Proposition \ref{P:quotient}, the algebra
  $B$ belongs to $\mcT_t \circ \mcS$.
\end{proof}

\subsection{The most basic case of $\mcT_t$}\label{S:LZS}
If we take $\mathcal G$ to be the variety of all groupoids  (magmas, binars) $(G, \cdot)$, and set
$t(x,y) = x \cdot y$, then $\mathcal G_t$ is the variety $\mc{LZ}$ of left-zero semigroups and
one obtains the following corollary.

\begin{corollary}\label{C:LZS}
The quasivariety $\mc{LZ} \circ \mc{S}$ is a variety. In particular $\mc{LZ} \circ \mc{S} = \pro{\mc{LZ}}$.
\end{corollary}

Note however that members of $\pro{\mc{LZ}}$ are not necessarily semigroups. First observe
that the equational base of $\pro{\mc{LZ}}$ contains the idempotent law $x \cdot x =
x$. The free $\pro{\mc{LZ}}$-algebra $A = X(\pro{\mc{LZ}})$ over the two element set
$X = \{x,y\}$ is the semilattice sum of three subalgebras: two one-element subalgebras
$A_x = \{x\}$ and $A_y = \{y\}$ and one subalgebra $A_{x,y}$ consisting of all elements
represented by terms of $T_2$. Obviously $xy, yx \in A_{x,y}$. To describe further
elements, let us introduce the following notation. For any element $a$ of $A_{x,y}$ and a
variable $z$, let $aR(z) = a \cdot z$ and $aL(z) = z \cdot a$. Then let $aE(z)$ be either
$aR(z)$ or $aL(z)$.
\begin{lemma}\label{L:fr}
Each element of $A_{x,y}$ different from $xy$ and $yx$ may be expressed in the standard form
\begin{equation}\label{E:lz}
(xy)E(z_1) \dots E(z_{k}) \,\, \mbox{or} \,\, (yx)E(z_1) \dots E(z_{k}),
\end{equation}
where $z_i = x$ or $z_i = y$, and k = 1, 2, ....
\end{lemma}
\begin{proof}
The proof is by induction on the length of the element. The shortest elements $xy$ and
$yx$ are already in $A_{x,y}$.
Since $A_{x,y}$ is a left-zero band, it follows that for any $u, v \in A_{x,y}$, we have $u \cdot v = u$.
So longer elements of $A_{x,y}$ may be obtained from elements of the form \eqref{E:lz}
only by multiplying them by a variable $x$ or a variable $y$ from the left or from the
right.
\end{proof}

\begin{corollary}\label{C:freePLZ}
The free algebra $A = X\pro{\mc{LZ}}$ is not a semigroup.
\end{corollary}
\begin{proof}
By Lemma \ref{L:fr}, it is easy to see that the elements $x \cdot yx$ and $xy \cdot x$ are different.
\end{proof}

On the other hand, the Mal'tsev product $\mc{LZ} \circ_{\mc{SG}} \mcS$ of the varieties
$\mc{LZ}$ and $\mcS$, but taken relative to the variety $\mc{SG}$ of semigroups, is a
variety of semigroups, and may be described in a simple way. Note that $\mc{LZ}
\circ_{\mc{SG}} \mcS$ is just the intersection $\pro{\mc{LZ}} \cap \mc{SG}$, and that the
members of this variety are bands (idempotent semigroups).

\begin{proposition}\label{P:LZSSG}
The Mal'tsev product $\mc{LZ} \circ_{\mc{SG}} \mcS$ coincides with the variety $\mc{LR}$ of left-regular bands defined by the identity $xyx = xy$.
\end{proposition}
\begin{proof}
$(\Rightarrow)$
Let $A$ be a member of $\mc{LZ} \circ_{\mc{SG}} \mcS$. Then $A$ is a band and a semilattice of left-zero bands. For $a, b \in A$, we have
\begin{equation*}
ab/\varrho = a/\varrho \cdot b/\varrho = b/\varrho \cdot a/\varrho = ba/\varrho.
\end{equation*}
Hence $ab \mathrel{\varrho} ba$, so $ab \cdot ba = ab$. On the other hand $ab \cdot ba = abba = aba$, which implies $aba = ab$.

$(\Leftarrow)$
Now assume that $A \in \mc{LR}$. By McLean's Theorem \cite{Mc54}, each band is a
semilattice of rectangular bands. Hence $A = \bigsqcup_{s \in S} A_s$, where $S$ is the
semilattice replica of $A$ and all $A_s$ are rectangular bands, defined by the identity
$xyx = x$. This together with the identity $xyx = xy$ gives $xy = x$. Hence all $A_s$ are
left-zero bands.
\end{proof}

\subsection{Mal'tsev products of some regular varieties}\label{S:CGS}
We now demonstrate that the assumption of strong irregularity of the variety $\mcT_t$ in
Theorem~\ref{T:varPt} is essential.

Consider again the type $\tau$ with one binary multiplication $\cdot$. Let $\mc{CG}$ be the variety of commutative groupoids.

\begin{proposition}\label{P:CGS quasi}
The quasivariety $\mc{CG} \circ \mcS$ satisfies the quasi-identitity
\begin{equation}\label{E:qcg}
(zx = x \mathrel{\&} zy = y \mathrel{\&} xz = yz) \rightarrow (xy = yx).
\end{equation}
\end{proposition}
\begin{proof}
Let $A\in \mc{CG} \circ \mcS$ and let $\varrho$ be the semilattice replica congruence of $A$. Assume that
$a, b, c \in A$ with $ca = a$, $cb = b$, and $ac = bc$. We will show that $ab = ba$.
Since $A/\varrho$ is a semilattice, it is commutative. Hence $ac \mathrel{\varrho} ca$
and $bc \mathrel{\varrho} cb$. Consequently
\begin{equation*}
a = ca \mathrel{\varrho} ac = bc \mathrel{\varrho} cb = b.
\end{equation*}
This implies
\begin{equation*}
ab = ba,
\end{equation*}
as desired.
\end{proof}

\begin{example}\label{Ex:SS}
Let $A$ be the groupoid whose multiplication table is given below.
\begin{center}
    \begin{tabular}{c|ccccccc}
   $\cdot$ &0&1&2&3&4&5&6\\
    \hline
    0& 0 & 0 & 0 & 0 & 0 & 0 & 0\\
    1& 0 & 1 & 0 & 0 & 0 & 0 & 0\\
    2& 0 & 0 & 2 & 2 & 0 & 0 & 2\\
    3& 0 & 0 & 2 & 3 & 0 & 1 & 2\\
    4& 0 & 0 & 0 & 0 & 4 & 4 & 4\\
    5& 0 & 0 & 0 & 0 & 4 & 5 & 4\\
    6& 0 & 0 & 2 & 3 & 4 & 5 & 6\\
    \end{tabular}
\end{center}
Let us note that the elements $0, 1, 2, 3, 4$ form a subsemilattice of $A$ (that means a
subgroupoid of $A$ which is a semilattice). Similarly the elements $0, 1, 2, 4, 5$ form a
subsemilattice, and the elements $0, 1, 2, 4, 6$ form a subsemilattice. However $3 \cdot 5
= 1 \neq 0 = 5 \cdot 3$, whence $A$ itself is not a semilattice. The equivalence $\varrho$
of $A$ with the $2$-element  classes $\{0, 1\}$, $\{2, 3\}$, $\{4, 5\}$, and one element
class $\{6\}$, is a congruence of $A$, each $\varrho$-class is a semilattice and
$A/\varrho$ is a semilattice. Thus $A \in \mathcal S \circ \mathcal S$, and hence $A \in
\mc{CG} \circ \mcS$.
Moreover, $\varrho$ is the semilattice replica congruence of $A$. Indeed, if $A/\sigma$ is
a semilattice for some congruence $\sigma$ of $A$, then for all $a, b \in A$ we must have
$(a \cdot b, \, b \cdot a) \in \sigma$. Hence
\begin{align*}
&(0, 1) = (5 \cdot 3, 3 \cdot 5) \in \sigma,\\
&(2, 3) = (3 \cdot 6, 6 \cdot 3) \in \sigma,\\
&(4, 5) = (5 \cdot 6, 6 \cdot 5) \in \sigma.
\end{align*}
Thus $\varrho$ is the semilattice replica congruence.

Now let $\theta$ be the congruence on $A$ generated by $2$ and $4$. One easily checks
that the congruence $\theta$ has one $3$-element class $\{0, 2, 4\}$ and each of the
remaining classes consists of one element. Let $B =A/\theta$.
By taking
$x=3/\theta=\{3\}$, $y= 5/\theta=\{5\}$ and $z=6/\theta = \{6\}$ we see that $B$ fails the
quasi-identity~\eqref{E:qcg}. By Proposition~\ref{P:CGS quasi}, $B \notin \mc{CG} \circ
\mcS$, and hence $B \notin \mcS \circ \mcS$. Consequently, neither $\mcS \circ \mcS$ nor
$\mc{CG} \circ \mcS$ is a variety.
\end{example}

From  Proposition~\ref{P:CGS quasi}, Example~\ref{Ex:SS}, and the fact that $\mcS$ is the
smallest regular variety, we have the following.

\begin{corollary}
For no regular subvariety $\mcV$ of $\mc{CG}$ is $\mcV\circ\mcS$ a variety.
\end{corollary}

By Corollary~\ref{C:vp_base}, the prolongation $\pro{\mathcal{CG}}$ is defined by the set of identities
$\Sigma^p$ consisting of all identities of the form
\begin{equation}
  \label{E:PC}
  (x_1\ldots x_m r_1) \ \cdot \ (x_1\ldots x_m r_2) =  (x_1\ldots x_m r_2) \ \cdot \ (x_1\ldots x_m r_1),
\end{equation}
in which $m > 0$ and $r_1, r_2$ are terms from $T_m$. Our discussion above suggests the
following problem.

\begin{problem}
Is the quasivariety $\mathcal{CG} \circ \mcS$ axiomatized by the identities in~\eqref{E:PC}
together with the single quasi-identity \eqref{E:qcg} of Proposition~\ref{P:CGS quasi}?
\end{problem}

Presumably, there is nothing special about commutativity here. Thus we are motivated to
pose the following problem.

\begin{problem}
 Let $\mcV$ be a proper regular variety of groupoids. Is it true that $\mcV \circ \mcS$
 fails to be a variety?
\end{problem}

With regard to this latter problem, let us note that, with $\mathcal G$ denoting the variety
of all groupoids, $\mathcal G \circ \mcS =
\mathcal G$, which is obviously a variety. Thus the requirement that $\mcV$ be a proper subvariety is necessary.

\section{The Mal'tsev product $\mcV_t \circ \mcS$}\label{S:Vt}

Now let $\mcV_t$ be a (strongly irregular) subvariety of the variety $\mcT_t$ of
$\Om$-algebras defined by the identity $ t(x,y) = x$ for some $t \in T_2$ and a set
$\Gamma_t$ of regular identities. Recall from Section \ref{S:subvarieties} that the
regularization $\tV_t$ of $\mcV_t$ is defined by $\Gamma_t$ and the identities (P1) --
(P5), with $x \cdot y = t(x,y)$, and consists of P\l onka sums of $\mcV_t$-algebras. Moreover
\begin{equation*}
\mcT_t \wedge \tV_t = \mcV_t \,\, \mbox{and} \,\, \mcT_t \vee \tV_t = \tT.
\end{equation*}
Since every P\l onka sum is a semilattice sum, $\tV_t \subseteq \mcV_t \circ \mcS$. On the
other hand,  $\mcV_t \circ \mcS$ is obviously contained in $\mcT_t \circ \mcS$. And by
Proposition~\ref{P:VoS}, $\mcV_t \circ \mcS \subseteq \pro{\mcV_t}$.  Finally, by
Theorem~\ref{T:varPt}, $\mcT_t \circ \mcS = \pro{\mcT_t}$. Summarizing
\begin{equation*}
\mcV_t \, \subseteq \, \tV_t \, \subseteq \, \mcV_t \circ \mcS \, \subseteq \pro{\mcV_t}
\, \subseteq \, \mcT_t \circ \mcS = \pro{\mcT_t} .
\end{equation*}
And by Corollary \ref{C:gen}, the variety $\pro{\mcV_t}$ is the smallest variety
containing the quasivariety $\mcV_t \circ \mcS$.

Recall also that, by \eqref{E:replP},
\begin{equation*}
 X\pro{\mcV_t}/\varrho \cong X\mcS,
\end{equation*}
and for any two $\Om$-terms $w, v$
\begin{equation*}
(w,v) \in \varrho \iff \mf{var}(w) = \mf{var}(v).
\end{equation*}

As a direct consequence of Proposition \ref{P:freetp} one obtains the following corollary.

\begin{corollary}\label{C:freePVt}
  For any subvariety $\mcV_t$ of $\mcT_t$, the free $\pro{\mcV_t}$-algebra
  $B = X\pro{\mcV_t}$ is a semilattice of $\mcV_t$-algebras. In particular,
  $B = \bigsqcup_{s \in S} B_s $, where $S = X\mcS$ is the free semilattice over $X$ and
  all $B_s$ are members of $\mcV_t$. Hence
  \begin{equation*}
B = X\pro{\mcV_t} \in \mcV_t \circ \mcS.
\end{equation*}
\end{corollary}

We intend to show that the classes $\mcV_t \circ \mcS$ and $\pro{\mcV_t}$ coincide.

\begin{proposition}\label{P:quotientC}
  Let $\delta$ be a congruence of the free $\pro{\mcV_{t}}$-algebra $B = X\pro{\mcV_t}$
  over~$X$. Then the quotient $C = B/\delta$ is a semilattice sum of $\mcV_t$-algebras.
\end{proposition}
\begin{proof}
We first use the fact that $C$ is the quotient $B/\delta$ of $B$ to describe the
decomposition of $C$ into a semilattice sum of $\mcT_t$-algebras in terms of
the algebra $B$.
Let $\varrho_B$ be the semilattice replica congruence of the algebra $B$, and let
$\psi = \delta \vee \varrho_B$. Since $\psi \geq \varrho_B$ it follows that $B/\psi$ is a
semilattice. On the other hand, the minimality of $\varrho_B$ guarantees that $\psi$ is
the smallest congruence above $\delta$ such that $B/\psi \in \mcS$. It follows that
$\psi/\delta$ is the semilattice replica congruence of $B/\delta$.
Since each element of $C$ is of the form $b/\delta$ for some $b\in B$,
Theorem~\ref{T:varPt} implies that
\begin{equation*}
  C = \bigsqcup_{b/\delta\in B/\delta} (b/\delta)/(\psi/\delta)
\end{equation*}
in which each summand $(b/\delta)/(\psi/\delta) \in \mcT_t$.

On the other hand, by Corollary~\ref{C:freePVt}, $B = \bigsqcup_{s \in S} B_s $, in which
$S$ is a semilattice and all $B_s$ are members of $\mcV_t$.

It remains to show that the $(\psi/\delta)$-delta classes, $(b/\delta)/(\psi/\delta)$,
satisfy the identities of $\Gamma_t$.  Let
\begin{equation}\label{E:gamma}
x_1 \ldots x_k u = x_1 \ldots x_k v
\end{equation}
be a typical member of $\Gamma_t$.
We want to show that \eqref{E:gamma} holds in each $\psi/\delta$-class. This means that
for any $c^1 = b^1/\delta, \ldots, c^k = b^k/\delta \in (b/\delta)/(\psi/\delta)$, with
$b^i \in B_{s_i}$ and $b \in B$, we have
\begin{equation}\label{E:ginC}
c^1 \ldots c^k u = c^1 \ldots c^k v,
\end{equation}
which is equivalent to
\begin{equation}\label{E:ginB}
(b^1 \ldots b^k u, b^1 \ldots b^k v) \, \in \, \delta.
\end{equation}
Note that $b^i/\delta, b^j/\delta \in (b/\delta)/(\psi/\delta)$ if and only if
$(b^i, b^j) \in \psi$. Since each $(b^i /\delta)/(\psi/\delta)$ satisfies the identity
$x \cdot y = x$, we may use Lemma~\ref{L:equivlz} repeatedly to obtain the following:
\begin{multline*}
b^{i}\ \delta\ b^{i}\cdot b^{1}\ \delta \ (b^{i}\cdot b^{1})\cdot b^{2}\ \delta\ ((b^{i}\cdot b^{1})\cdot b^{2})\cdot b^{3}\ \delta\dots\\
\delta\ (\ldots((b^{i}\cdot b^{1})\cdot b^{2})\cdot \ldots)\cdot b^{k}\ =: d^i.
\end{multline*}
Thus $b^{i}\ \delta\ d^i$ for each $1 \leq i \leq k$, and the elements
$d^1, d^2, \ldots, d^k$ all belong to the same congruence class $B_{s_{1}\cdots s_k}$ of
$\varrho_B$.  Since $\varrho_B$-classes satisfy the identities of $\Gamma_t$, it follows
that
\begin{equation*}
d^1 \ldots d^k u = d^1 \ldots d^k v.
\end{equation*}
Finally note that
\begin{equation*}
(b^1 \ldots b^k u, d^1 \ldots d^k u) \in \delta \quad \mbox{and} \quad (b^1 \ldots b^k v,
d^1 \ldots d^k v) \in \delta,
\end{equation*}
which implies
\begin{equation*}
(b^1 \ldots b^k u, b^1 \ldots b^k v) \in \delta.
\end{equation*}
This verifies \eqref{E:ginB}.
\end{proof}

Let us note that in particular, $\psi/\delta$ is the semilattice replica congruence of $B/\delta$ and the quotient $B/\delta$ is the semilattice sum of $\psi/\delta$-classes $(b/\delta)/(\psi/\delta)$.

\begin{theorem}\label{T:varqvar}
  Let $\mcV_t$ be a variety satisfying the identity $t(x,y)=x$ and with equational base
  $\Sigma$. Then the Mal'tsev product $\mcV_t \circ \mcS$ is a variety, having
  $\Sigma_{t}^{p}$ as an equational base.
\end{theorem}
\begin{proof}
  We have already observed that $\mcV_t \circ \mcS \subseteq \pro{\mcV_t}$. Conversely, if
  an algebra~$C$ is a member of $\pro{\mcV_t}$, then there is a set $X$ so that $C$ is a
  homomorphic image of $X\pro{\mcV_t}$. Then by Proposition \ref{P:quotientC}, the algebra
  $C$ belongs to $\mcV_t \circ \mcS$. Thus the quasivariety $\mcV_t \circ \mcS$ and the
  variety $\pro{\mcV_t}$ coincide.
\end{proof}

\begin{corollary}\label{C:relprod}
Let $\mc{W}$ be any variety of $\Omega$-algebras. Then the Mal'tsev product $\mcV_t \circ_{\mc{W}} \mcS$ of $\mcV_t$ and $\mcS$ relative to $\mc{W}$ is a variety.
\end{corollary}

\begin{problem}
Does Theorem \ref{T:varqvar} hold for an irregular variety which is not strongly irregular?
\end{problem}

\section{Examples and counterexamples}\label{sec:examples}

First consider the type $\tau$ with one binary multiplication $\cdot$, as in
Sections~\ref{S:LZS} and~\ref{S:CGS}. We could observe that the Mal'tsev products
$\mc{LZ} \circ \mcS$ and $\mcS \circ \mcS$ considered in these examples do not satisfy the
associative law. This is however an instance of a more general phenomenon. Consider a
variety of plural type $\tau$, and assume that the equational base $\Sigma$ of a
variety $\mcV$ contains a (non-trivial) linear identity $\sigma$. (An identity is
\emph{linear} if the multiplicities of each variable on each side are at most~$1$. This is
not to be confused with the notion of linear identity as it appears in~\cite{HM88}.) Then
the identities of $\sigma^p$ obtained from $\sigma$ are never linear, since each variable
of $\sigma$ was replaced in $\sigma^p$ by a term $r_i \in T_m$ with $m \geq 2$, and all of
these terms contain the same sets of variables. Moreover, consequences of the identities
$\Sigma^p$ are never linear.

This can be summarised as follows.

\begin{lemma}\label{L:linid}
For any non-trivial variety $\mcV$, the Mal'tsev product $\mcV \circ \mcS$ does not
satisfy any (non-trivial) linear identity.
\end{lemma}

In particular, for no variety $\mcV$ of semigroups is the Mal'tsev product
$\mcV \circ \mcS$ a class of semigroups.  A similar observation can be made for many
familiar varieties, for example groups, rings, lattices, semilattices etc.

Frequently more interesting are proper subvarieties of $\mcV \circ \mcS$.

\begin{example} [\textbf{Bands}] As already mentioned in Section~\ref{S:LZS}, each band is
  a semilattice sum of rectangular bands. (See \cite{Mc54} and \cite{H76}, \cite{H95}.)
  Let $\mc{SG}$ be the variety of semigroups, $\mc{B}$ the variety of bands, and
  $\mc{RB}$ the variety of rectangular bands. Then, by Corollary \ref{C:relprod},
  \begin{equation*}
\mc{B} = \mc{RB} \circ_{\mc{SG}} \mcS = \mc{SG} \cap (\mc{RB} \circ \mcS) = \mc{SG} \cap
\pro{\mc{RB}}. 
\end{equation*}
Note that $\pro{\mc{RB}}$ is a variety of groupoids, but is not a variety of semigroups.
\end{example}

In the following examples, the (pseudo-)partition operation $x \cdot y$ will be denoted by
$t(x,y)$ to avoid conflict with the basic operation of multiplication in the algebras
considered.

\begin{example} [\textbf{Birkhoff systems}]\label{ex:birkhoff}
Lattices may be defined as algebras $(A,+,\cdot)$ with two semilattice reducts, the
(join-)semilattice $(A,+)$ and the (meet-)semilattice $(A, \cdot)$, satisfying the
\emph{Birkhoff identity} $x + xy = x(x+y)$ and the \emph{absorption law} $t(x,y) = x + xy
= x$. They form the variety $\mc{L}$ of lattices. By dropping the absorption law one
obtains the definition of the more general variety $\mc{BS}$ of \emph{Birkhoff
  systems}. (See \cite{HR17a} and \cite{HR17b}.) An essential role in this variety is
played by two $3$-element bi-chains $\mbf{3}_m$ and $\mbf{3}_j$ built on the set $\{0, 1,
2\}$ by defining the meet-semilattice order as $0 <_{\cdot} 1 <_{\cdot} 2$ and the
join-semilattice order as $0 <_{+} 2 <_{+} 1$ in $\mbf{3}_m$ and as $1 <_+ 0 <_+ 2$ in
$\mbf{3}_j$. Let $\mc{BS}^-(\mbf3_m,\mbf3_j)$ denote the class of Birkhoff systems which do not
contain either $\mbf 3_m$ or $\mbf 3_j$ as subalgebras.
It was shown in \cite{HR17b} that the class $\mc{BS}^-(\mbf3_{m},\mbf3_{j})$ is
a variety, and is characterized as precisely the class of Birkhoff systems which are
semilattice sums of lattices. It follows that
\begin{equation*}
\mc{BS}^-(\mbf3_{m},\mbf3_{j}) = \mc{L} \circ_{\mc{BS}} \mcS = \mc{BS} \cap (\mc{L} \circ \mcS) =
\mc{BS} \cap \pro{\mc{L}}.
\end{equation*}
Note that none of the basic operations of $\pro{\mc{L}}$-algebras are associative or commutative.
\end{example}

As already mentioned, for a strongly irregular variety $\mcV_t$, the best known
and understood subvariety of $\pro{\mcV_t}$ (different from $\mcV_t$ and $\mcS$) is the
regularization, $\tV_t$, of $\mcV_t$. Much less is known about the pseudo-regularization. It
is an easy exercise to show that the regularization and the pseudo-regularization of the
variety $\mc{LZ}$ of left-zero bands coincide. However it is possible for the
regularization of a variety to be distinct from its pseudo-regularization. The first
example was discovered in \cite{BF15}, and investigated in connection with the constraint
satisfaction problem.

\begin{example}[\textbf{Steiner quasigroups}]\label{ex:steiner}
  A \emph{Steiner quasigroup} or \emph{squag} is a commutative idempotent groupoid
  satisfying the identity $t(x,y) = xy \cdot y = x$. The variety formed by these groupoids
  is denoted by $\mc{SQ}$. As shown in \cite{BF15}, the regularization $\wt{\mc{SQ}}$ of
  $\mc{SQ}$ is the variety of commutative and idempotent groupoids
  satisfying the identity $x (x \cdot yz) = (x \cdot xy) z$. The pseudo-regularization
  $\overline{\mc{SQ}}$ of $\mc{SQ}$ is defined to be the variety of commutative idempotent groupoids
  satisfying the identities (P1) -- (P4) from Section~$2$, where $\cdot$ is replaced by
  $t(x,y) = xy \cdot y$. Groupoids in the regularization are P\l onka sums of
  squags. Groupoids in the pseudo-regularizations are semilattice sums of squags. As was
  shown in \cite{BF15}, the variety of commutative idempotent groupoids
  satisfying the identity $x (y \cdot yz) = (xy \cdot y) z$ (called $ \mathcal{T}_2$ in
  that paper) is contained in the variety
  $\overline{\mc{SQ}}$, and is different from $ \wt{\mc{SQ}}$. It follows
  that the three varieties $\mc{SQ}$, $\wt{\mc{SQ}}$ and $\overline{\mc{SQ}}$ are
  distinct:
  \begin{equation*}
\mc{SQ} \, \subset \, \wt{\mc{SQ}} \, \subset \, \overline{\mc{SQ}}.
\end{equation*}
However we do not know if the varieties $\mathcal{T}_2$ and $\overline{\mc{SQ}}$ coincide,
though we think it is not likely. The Mal'tsev product $\mc{SQ} \circ \mcS$ coincides with
the variety $\pro{\mc{SQ}}$.  Since commutativity is a linear identity, members of the
variety $\pro{\mc{SQ}}$ are not commutative. The pseudo-regularization
$\overline{\mc{SQ}}$ is obviously contained in but not equal to $\pro{\mc{SQ}}$.
\end{example}

In fact, Examples~\ref{ex:birkhoff} and~\ref{ex:steiner} were the inspiration for
undertaking investigations of the most general class of algebras representable as
semilattice sums of their subalgebras.

Examples of other varieties of algebras with distinct regularization and
pseudo-regularization are provided by bisemilattices, algebras with two semilattice
operations (one interpreted as a join and one as a meet, similarly as in the case of
Birkhoff systems.) 

\begin{example}[\textbf{Semilattice sums of lattices}]
  First recall that members of the class $\mc{BS}^-(\mbf3_{m},\mbf3_{j})$ of Birkhoff
  systems considered above are all semilattice sums of lattices. However, it is
  not difficult to check that this class is not the pseudo-regularization of the variety
  $\mc{L}$ of lattices. The partition operation $t(x,y) = x + xy$ of the regularization
  $\wt{\mcL}$ is not a pseudo-partition operation for
  $\mc{BS}^-(\mbf3_{m},\mbf3_{j})$. Indeed, an easy exercise shows that the requirements for
  a pseudo-partition operation are not satisfied. A witness is a certain $3$-element
  member of $\mc{BS}^-(\mbf3_{m},\mbf3_{j})$. This is the bichain $\mbf{3}_n$ built on the
  set $\{0, 1, 2\}$ with the meet-semilattice order defined as for $\mbf{3}_m$ and
  $\mbf{3}_j$ and the join-semilattice order defined by $2 <_+ 0 <_+ 1$. Note that
  $\mbf{3}_n$ is the semilattice sum of the $2$-element lattice $\{0, 1\}$ and the
  $1$-element lattice $\{2\}$, but it is not a P\l onka sum of these lattices. (See
  \cite{HR17a} and \cite{HR17b}.)

Now let $ A$ be a subdirectly irreducible lattice and set $A_1=A$. Let $A_0$ be a $1$-element
lattice $\{\infty\}$ disjoint from $A_1$. Finally let $A^{\infty}$ be the P\l onka
sum of $A_1$ and $A_0$ over the $2$-element meet semilattice $\{0, 1\}$ with $0 < 1$. Then
the bisemilattice $A^{\infty}$ generates the regularization of the variety $\mathsf{V}(A)$
generated by $A$. (See e.g. \cite[Ch.~4]{RS02}.)

There is another interesting family of semilattice sums obtained from the lattice $A$ in
the case when $A$ is a bounded lattice. Once again set $A_1 = A$ and denote the bounds of
$A$ by $0_A$ and $1_A$. We define the bisemilattice $A_{\infty} = A \cup \{\infty\}$ as
follows.  For an $a \in A$, $a\cdot \infty = 0_A$ and $a+\infty = 1_A$.  It is easy to see
that this yields a bisemilattice, and is a semilattice sum of the lattices $A_1$ and $A_0$
over the meet semilattice $\{0, 1\}$ with $1 < 0$. A somewhat more complicated exercise
shows that the operation $t(x,y) = x + xy$ is a pseudo-partition operation on the algebra
$A_{\infty}$, but it is not a partition operation. It follows that each variety
$\mathsf{V}(A_{\infty})$ is contained in the pseudo-regularization of $\mc{L}$ but is
different from the regularization of the variety $\mathsf{V}(A)$. We do not know if the
varieties $\mathsf{V}(A_{\infty})$ and $\overline{\mathsf{V}(A)}$ coincide.
\end{example}

\begin{example}[\textbf{Groups}]
  Groups may be defined as algebras $(G, \cdot, ^{-1})$ satisfying the identities
  $xy \cdot z = x \cdot yz$, $y^{-1}y \cdot x = x \cdot y y^{-1}$, and
  $x \cdot y y^{-1} = x$. They form a strongly irregular variety with
  $t(x,y) = x \cdot y y^{-1}$, which will be denoted by $\mc{GP}$. The regularization
  $\wt{\mc{GP}}$ of $\mc{GP}$ may be defined by the first two identities defining $\mc{GP}$ and
  the identities $(xy)^{-1} = y^{-1} x^{-1}, \, (x^{-1})^{-1} = x$ and $xxx^{-1} =
  x$. (See e.g. \cite{P69} and \cite{PR92}.) The variety $\wt{\mc{GP}}$ consists of P\l
  onka sums of groups. In semigroup theory, P\l onka sums of groups are called
  \emph{strong semilattices of groups}. By results of Clifford~\cite{C41}, it is known
  that the class of P\l onka sums of groups coincides with the class of semigroups which
  are semilattices of groups and with the class of so-called \emph{Clifford semigroups},
  semigroups with a unary operation $^{-1}$ satisfying the identities
  $(x^{-1})^{-1} = x, xx^{-1}x = x, x x^{-1} = x^{-1} x$ and
  $(x x^{-1}) (y y^{-1}) = (y y^{-1}) (x x^{-1})$.  (See also \cite{CP61}, \cite{H76} and
  \cite{H95}.)  The class $\wt{\mc{GP}}$ is a subvariety of the variety $\mc{IS}$ of
  inverse semigroups, semigroups with one unary operation $^{-1}$ satisfying the last
  three identities defining $\wt{\mc{GP}}$. Members of $\wt{\mc{GP}}$ are precisely inverse
  semigroups which are P\l onka sums of groups (or as semigroup theorists would say,
  strong semilattices of groups). (See e.g. \cite{PR92} and references there.) In fact,
  \begin{equation*}
\wt{\mc{GP}} = \mc{IS} \cap (\mc{GP} \circ \mcS).
\end{equation*}
Groups defined as algebras with two basic operations $\cdot$ and $^{-1}$ as above do not
contain any nullary operations in the similarity type. Hence Theorem~\ref{T:varqvar} applies, and the
Mal'tsev product $\mc{GP} \circ \mcS$ is the variety $\pro{\mc{GP}}$. Note that
$\pro{\mc{GP}}$ is a variety of groupoids with one unary operation, but it is not a variety
of semigroups. The pseudo-regularization of $\mc{GP}$ has not yet been investigated.
\end{example}

The assumption that the type of a variety $\mcV$, in a Mal'tsev product $\mcV \circ \mcS$,
has no nullary operation symbols is due to the fact that semilattices have no nullary
operations. However, if $\mcV$ is a variety with constant (nullary) operations, then the
latter assumption may easily be omitted, as it was done in the example above. Instead of a
constant operation one may take a unary operation defining this constant, and consider an
equivalent variety $\mcV'$ with such a unary operation replacing each constant
operation. Then if $\mcV$ is strongly irregular, Theorem \ref{T:varqvar} applies, and the
Mal'tsev product $\mc{V'} \circ \mcS$ is a variety.

If $\mcV$ is an irregular, but not strongly irregular variety,
then the class of P\l onka sums of $\mcV$-algebras does not necessarily coincide with the
regularization $\wt{\mcV}$ of $\mcV$. (See \cite[Ch.~4]{RS02} and \cite{PR92} for more
information and some references.) The following example concerns the case of semigroups.

\begin{example}[\textbf{Semigroups}]
  It was shown by V. N. Sali\v{\i} \cite{VS69, VS71} that each semigroup in the
  regularization $\wt{\mcV}$ of an irregular variety $\mcV$ of semigroups embeds into a
  P\l onka sum of $\mcV$-algebras. (This however is not true for general algebras, see
  \cite{GKW86} and \cite{R86}.) Hence in the case of semigroups the
  regularization $\wt{\mcV}$ is contained in the Mal'tsev product $\mcV \circ
  \mcS$. However not all of its members are P\l onka sums of $\mcV$-algebras.
\end{example}

\begin{example}
  This example comes from the geometry of affine spaces and convex sets. (See \cite{RS02}
  and references there.) Affine spaces over a subfield $F$ of the field $\mb{R}$ of reals
  (affine $F$-spaces) may be defined as algebras $(A, \underline{F})$, where
  $\underline{F}$ is the set of binary operations
  $xy\underline{p} = \underline{p}(x,y) = x(1-p) + yp$, for all $p \in F$. Affine $F$-spaces
  defined in this way form a variety. Convex subsets of affine $F$-spaces (convex
  $F$-sets) may be defined as algebras $(C, \underline{I}^o)$, where
  $\underline{I}^o = \bigl\{\underline{p} \mid p \in I^o = ]0,1[\bigr\}$ and $]0,1[$
  is the open unit interval of $F$. Fix a subfield $F$ of $\mb{R}$. Then the convex
  $F$-subsets of affine $F$-spaces generate the variety $\mc{BA}$ of \emph{barycentric
    algebras}. Convex $F$-sets form a subquasivariety $\mc{C}$ of $\mc{BA}$.  The variety
  $\mc{BA}$ is defined by the following identities
\begin{align*}
& x x \underline{p} = x, \\
& x y\, \underline{p} = y x\, \underline{1-p}, \\
& x y\, \underline{p}\ z\, \underline{q} = x\,\, y z\,
\underline{q/(p\circ q)}\,\, \underline{p\circ q}
\end{align*}
for all $p, q$ in $I^{o}$. Here $p\,
\circ\, q = p+q-pq$.
(See \cite[\S5.8]{RS02}.)
Its subquasivariety $\mc{C}$ is defined by the cancellation laws
\begin{equation}\label{E:cl}
(xy\underline{p} = xz\underline{p}) \rightarrow (y = z),
\end{equation}
which hold for all $p \in I^{o}$.  It is known \cite[\S7.5]{RS02} that each barycentric
algebra is a semilattice sum of (open) convex sets. This shows that $\mc{BA}$ is contained
in the Mal'tsev product $\mc{C} \circ \mcS$:
\begin{equation*}
\mc{BA} \subset \mc{C} \circ \mcS \subset \pro{\mc{BA}}. 
\end{equation*}
Since the quasivariety $\mc{C}$ satisfies linear identities, the varieties $\mc{BA}$ and
$\pro{\mc{BA}}$ are distinct.

In fact, a stronger result was proved, which we provide here in a somewhat reformulated way.

\begin{theorem}\cite[\S 7.5]{RS02}
Each barycentric algebra is a semilattice sum of open convex sets over its semilattice
replica, and is a subalgebra of a P\l onka sum of convex sets over its semilattice
replica.
\end{theorem}
\noindent The summands of the P\l onka sum are certain canonical extensions of the
summands of the semilattice sum.

\end{example}

\section{Concluding remarks and the problem of representation of algebras in Mal'tsev products}

The fact that an algebra $A$ is a semilattice sum of $\mcV_t$-subalgebras says very little
about its detailed structure. An exception is of course given by algebras represented as
P\l onka sums. (The second best representation is given by subalgebras of P\l onka sums.)
So it is natural to ask how the summands of a semilattice sum are put together or how to
reconstruct the algebraic structure of a semilattice sum $A = \bigsqcup_{s \in S} A_s$
from its $\mcV_t$-summands $A_s$ and the quotient semilattice $S$. Such a construction
exists and was introduced for general algebras in \cite[\S6.2]{RS85}, under the name of a
Lallement sum, as a generalization of a P\l onka sum. (See also \cite{L67} for a basic
construction introduced in the case of semigroups, and \cite{R86}, \cite{PR92} and
\cite[Ch.~4]{RS02} for similar constructions in the case of general algebras.) The primary
idea was to relax the \mbox{requirement} of functoriality in the definition of P\l onka
sums. There are several types of such constructions and the definition may be formulated
for algebras of any plural type. However, to avoid excessive notation we will consider
here only plural type with operations of arity two, and will limit ourself to the
following definition.

\begin{definition}\label{D:LS}
Let $(S, \cdot)$ be a (meet) semilattice. (Here all operations of $\Om$ are equal to
$\cdot$). For each $s \in S$, let a $\mcV_t$-algebra $A_s$ of a plural binary type
$\tau\colon \Om \rightarrow \{2\}$ be given, and for each operation $\star \in \Om$ an
extension $(E_{s}^{\,\star},\star)$ of $(A_{s},\star)$. For each pair $t \leq s$ of $S$,
let 
\begin{equation*}
\varphi_{s,t}^{\,\star}\colon (A_s, \star) \rightarrow (E_{t}^{\,\star},\star)
\end{equation*}
be a $\star$-homomorphism
such that the following three conditions are satisfied:
\begin{enumerate}
\item $\varphi_{s,s}^\star$ is the embedding of $A_s$ into $E_{s}^{\star}$;
\item $\varphi_{s,s\cdot t}^{\star}(A_s) \star \varphi_{t,s\cdot t}^{\star}(A_t) \subseteq A_{s\cdot t}$;
\item for each $u \leq s\cdot t$ in $S$ and $a_s \in A_s,\, b_t \in A_t$
  \begin{equation*}
\varphi_{s\cdot t,u}^\star (\varphi_{s,s\cdot t}^\star (a_s) \star \varphi_{t,s\cdot
  t}^\star (b_t)) = \varphi_{s,u}^\star(a_s) \star \varphi_{t,u}^\star(b_t). 
\end{equation*}
\end{enumerate}

\noindent An $\Om$-algebra structure on the disjoint sum $A$ of all $A_s$ is given by defining the operations $\star$ in $\Om$ as follows:
\begin{equation*}
a_s \star b_t = \varphi_{s,s\cdot t}^{\,\star}(a_s) \star \varphi_{t,s\cdot t}^{\,\star}(b_t).
\end{equation*}
Then all $A_s$ are subalgebras of $A$, and $S$ is a quotient of $A$. The semilattice sum
$A$ of $A_s$ is said to be the \emph{semilattice sum of $A_s$ by the mappings
  $\varphi_{s,t}^{\star}$}. If additionally, for each $t \in S$, one has
$E_{t}^\star = \{\varphi_{s,t}^{\star}(a_s) \mid s \geq t,\, a_s \in A_s\}$, and all
$E_{s}^{\star}$ are certain canonical extensions of $A_s$ (see \cite[\S~6.1]{RS85}), then
this semilattice sum is called a \emph{Lallement sum}.
\end{definition}

By \cite[Th.~624]{RS85} each semilattice sum of $\mcV_t$-algebras can be reconstructed as
a Lallement sum of these algebras. The usefulness of Lallement sums depends on properties
of the available extensions $E_{s}^\star$. In particular, a nice situation appears if for
each $s \in S$, all extensions $E_s^{\star}$ coincide with the summand $A_s$. Such
Lallement sums are called \emph{strict}.

We next consider a special case of strict Lallement sums of $\mcV_t$-algebras, which
extends a construction described for Birkhoff systems in \cite{HR17b}.  This construction
works nicely for $\mcV_t$-algebras where each operation $\star \in \Om$ has a one-sided
unit. In what follows we assume that all one-sided units are right-sided and call them
simply units. We will call such algebras briefly \emph{$\mcV_t$-algebras with units}.

Let $A$ be a semilattice sum $\bigsqcup_{s\in S} A_s$ of $\mcV_t$-algebras $A_s$ with
units, and let $e_{s}^{\star}$ be the (right-)unit element of $\star$ in
$A_s$. Additionally assume that for all $s, t, u \in S$ with $u \leq s, t$
\begin{equation}\label{E:const}
(a_s \star b_t) \star e_{u}^{\star} = (a_s \star e_{u}^{\star}) \star (b_t \star e_{u}^{\star}),
\end{equation}
and for $t\leq s$ in $S$ define the maps
\begin{equation*}
\varphi_{s,t}^{\star} \colon A_s \rightarrow A_t;\  a_s \mapsto a_s \star e_{t}^{\star}.
\end{equation*}
Note that $t\leq s$ in $S$ means $s\cdot t = t$, so these maps are well
defined. It is easily seen that each $\varphi_{s,t}^{\star}$ is a $\star$-homomorphism
from $A_s$ into $A_t$. Moreover, by \eqref{E:const}, 
\begin{equation*}
a_s \star b_t = (a_s \star b_t) \star e_{st}^{\star} = (a_s \star e_{t}^{\star}) \star (b_s \star e_{t}^{\star}) =
\varphi_{s,s\cdot t}^{\star}(a_s) \star \varphi_{t,s\cdot t}^{\star}(b_t).
\end{equation*}
Then for each $u \leq s\cdot t$ in $S$
\begin{align*}
&\varphi_{s\cdot t,u}^\star (\varphi_{s,s\cdot t}^\star (a_s) \, \star \, \varphi_{t,s\cdot t}^\star (b_t)) =
\varphi_{s\cdot t,u}^\star (a_s \, \star \,  b_t) =\\
 &(a_s \star b_t) \star e_{u}^{\star} = (a_s \star e_{u}^{\star}) \star (b_t \star e_{u}^{\star}) =
\varphi_{s,u}^\star(a_s) \star \varphi_{t,u}^\star(b_t).
\end{align*}
It follows that the maps $\varphi_{s,t}^{\star}$ satisfy the requirements of a strict
Lallement sum.  One obtains the following theorem, a corollary of Theorem~624 of
\cite{RS85}, and an extension of Theorem $4.15$ of \cite{HR17b}.

\begin{theorem}\label{T:LS}
  Let $A$ be a semilattice sum of $\mcV_t$-algebras $A_s$ with units of a plural binary
  type. Then $A = \bigsqcup_{s \in S} A_s$ satisfies the condition \eqref{E:const} if and
  only if it is a strict Lallement sum of the subalgebras $A_s$ over the semilattice $S$
  given by the homomorphisms $\varphi_{s,t}^{\star}$ described above.
\end{theorem}

Examples are provided by Birkhoff systems which are semilattices of bounded lattices, and
more general bisemillatices which are semilattice sums of bounded lattices. In particular,
members of the pseudo-regularizations of the variety $\mc{L}$ of lattices provided above
are of this type.

Note also that the condition \eqref{E:const} holds for all algebras $A$ of plural binary
type with self-entropic operations, i.e.,  satisfying $(x \star y) \star (z \star t)
=  (x \star z) \star (y \star t)$, for each $\star \in \Om$.

One more final remark concerns the assumption about strong irregularity of the summands
$A_s$. In this paper, we were interested in semilattices sums of algebras in a strongly
irregular variety. However, both Definition \ref{D:LS} and Theorem \ref{T:LS} may be
extended easily to semilattice sums of algebras in any variety.

\bigskip
\noindent\textbf{Remark.} The authors thank the referee for suggesting the term
`prolongation' as used in this paper.


\begin{thebibliography}{99}

\bibitem{B18}
Bergman, C.:
Notes on Quasivarieties and Mal'tsev Products.
\url{https://faculty.sites.iastate.edu/cbergman/files/inline-files/maltsevprods.pdf}

\bibitem{CB12}
Bergman, C.:
Universal algebra. Fundamentals and selected topics.
CRC Press, Boca Raton (2012)

\bibitem{BF15}
Bergman, C., Failing D.:
Commutative idempotent groupoids and the constraint satisfaction problem.
Algebra Universalis \textbf{73}, 391--417 (2015)


\bibitem{BR96}
Bergman, C., Romanowska, A.:
Subquasivarieties of regularized varieties.
Algebra Universalis \textbf{36}, 536--563 (1996)


\bibitem{C41}
Clifford, A.H.:
Semigroups admitting relative inverses.
Ann. Math. \textbf{42}, 1037--1049 (1941)

\bibitem{CP61}
Clifford, A.H., Preston, G.B.:
The Algebraic Theory of Semigroups.
Amer. Math. Soc., Providence (1961)

\bibitem{GKW86}
Graczy\'{n}ska, E., Kelly, D., Winkler, P.:
On the regular part of varieties of algebras.
Algebra Universalis \textbf{23}, 77--84 (1986)


\bibitem{HR17a}
Harding, J., Romanowska, A.:
Varieties of Birkhoff systems, Part I.
Order \textbf{34}, 45--68 (2017)

\bibitem{HR17b}
Harding, J., Romanowska, A.:
Varieties of Birkhoff systems, Part II.
Order \textbf{34}, 69--89 (2017)

\bibitem{HM88}
Hobby, D., McKenzie, R.:
The structure of finite algebras. Contemporary Mathematics, vol. 76.  
 American Mathematical Society, Providence  (1988)


\bibitem{H76}
Howie, J.M.:
An Introduction to Semigroup Theory.
Academic Press, London (1976)

\bibitem{H95}
Howie, J.M.:
Fundamentals of Semigroup Theory.
Clarendon Press, Oxford (1995)


\bibitem{I84}
Iskander, A.A.:
Extension of algebraic systems.
Trans. Amer. Math. Soc. \textbf{28}, 309--327 (1984)


\bibitem{L67}
Lallement, G.:
Demi-groupes r\' {e}guliers.
Ann. Mat. Pura Appl. \textbf{77}, 47--129 (1967)

\bibitem{M67}
Mal'tsev, A.I.:
Multiplication of classes of algebraic systems. 
Sibirsk. Mat. Zh. \textbf{8}, 346--365 (1967) (Russian).
English translation in: The Metamathematics of Algebraic Systems. Collected Papers: 1936--1967

\bibitem{M71}
Mal'tsev, A.I.:
The Metamathematics of Algebraic Systems. Collected Papers: 1936--1967.
Noth-Holland Publishing Co., Amsterdam (1971)

\bibitem{MMT87}
McKenzie, R., McNulty, G., Taylor, W.:
Algebras, lattices, varieties, volume 1.
Wadsworth and Brooks/Cole, Monterey (1987)

\bibitem{Mc54}
McLean, D.:
Idempotent semigroups.
American Mathematical Monthly \textbf{61}, 110--113 (1954)

\bibitem{P67}
P\l onka, J.:
On a method of construction of abstract algebras.
Fund. Math. \textbf{60}, 183--189 (1967)

\bibitem{P69}
P\l onka, J.:
On equational classes of abstract algebras defined by regular equations.
Fund. Math \textbf{64}, 241--247 (1969)


\bibitem{PR92}
P\l onka, J., Romanowska, A.:
Semilattice sums.
in: Romanowska, A., Smith, J.D.H. (eds.)
Universal Algebra and Quasigroup Theory, pp. 123--158.
Heldermann, Berlin (1992)

\bibitem{R86}
Romanowska, A.:
On regular and regularized varieties.
Algebra Universalis \textbf{23}, 215--241  (1986)

\bibitem{RS85}
Romanowska, A., Smith, J.D.H.:
Modal theory.
Heldermann Verlag, Berlin (1985)

\bibitem{RS02}
Romanowska, A., Smith, J.D.H.:
Modes.
World Scientific, Singapore (2002)

\bibitem{VS69}
Sali\v{\i}, V.N.:
Equationally normal varieties of semigroups. 
Izv. Vyssh. Uchebn. Zaved. Mat. \textbf{84}, 61--68  (1969) (Russian),

\bibitem{VS71}
Sali\v{\i}, V.N.:
A theorem on homomorphisms of strong semilattices of semigroups. 
In: Vagner, V.V. (ed.) Theory of Semigroups and Applications 2 pp. 69--74.
Izd. Saratov. Univ., Saratov (1971). (Russian)


\end{thebibliography}
\end{document}